\documentclass[12pt,reqno]{amsart}

\newcommand\version{March 22, 2022}

\usepackage{amsmath,amsfonts,amsthm,amssymb,amsxtra}
\usepackage{bbm} % to get \1

\newtheorem{theorem}{Theorem}%[section]

\newtheorem{lemma}[theorem]{Lemma}

\theoremstyle{definition}

\theoremstyle{remark}

\newcommand{\1}{\mathbbm{1}}

\renewcommand{\epsilon}{\varepsilon}

\renewcommand{\phi}{\varphi}
\newcommand{\R}{\mathbb{R}}

\newcommand{\Sph}{\mathbb{S}}

\newcommand{\Z}{\mathbb{Z}}

\DeclareMathOperator{\spec}{spec}

\DeclareMathOperator{\Tr}{Tr}

\begin{document}

\title[Two consequences of Davies's Hardy inequality --- \version]{Two consequences of Davies's Hardy inequality}

\author[R.~L.~Frank]{Rupert L. Frank}
\address[R.~L.~Frank]{Mathematisches Institut, Ludwig-Maximilans Universit\"at M\"unchen, Theresienstr.~39, 80333 M\"unchen, Germany, and Department of Mathematics, California Institute of Technology, Pasadena, CA 91125, USA}
\email{r.frank@lmu.de, rlfrank@caltech.edu}

\author[S.~Larson]{Simon Larson}
\address[S.~Larson]{Department of Mathematics, California Institute of Technology, Pasa\-de\-na, CA 91125, USA}
\email{larson@caltech.edu}

\dedicatory{In memory of M.~Z.~Solomyak, on the occasion of his 90th birthday}

\thanks{\copyright\, 2020 by the authors. This paper may be reproduced, in its entirety, for non-commercial purposes.\\
	U.S.~National Science Foundation grants DMS-1363432 and DMS-1954995 (R.L.F.) and Knut and Alice Wallenberg Foundation grant KAW~2018.0281 (S.L.) are acknowledged.\\
	In March 2022, Yi Huang made us aware of the fact that Theorem \ref{lieb} and its derivation from Davies's Hardy inequality are known and appear as \cite[Theorem 1.5.11]{Da0}. Moreover, he pointed out a typographical error in the definition of $\delta(x)$, which propagated through the paper and which is corrected here. We thank Yi Huang for these comments.}

\maketitle

\section{Introduction}

In this short note we would like to show that one can use Davies's Hardy inequality to rederive well-known results of Lieb \cite{Li} and Rozenblum \cite{Ro}. Throughout the following we fix an open set $\Omega\subset\R^d$ and define, for $\omega\in\Sph^{d-1}$,
$$
\delta(x) := \left( d \left| \Sph^{d-1}\right|^{-1} \int_{\Sph^{d-1}} d_\omega(x)^{-2}\,d\omega \right)^{-1/2}
\quad\text{where}\qquad
d_\omega(x) := \inf\{ |t|:\ x+t\omega\notin\Omega \}
$$
(with the convention that $\inf\emptyset = 0$). Then Davies's Hardy inequality \cite{Da} states that
\begin{equation}
	\label{eq:davies}
	\int_\Omega |\nabla u|^2\,dx \geq \frac14 \int_\Omega \delta^{-2}|u|^2\,dx
	\qquad\text{for all}\ u\in H^1_0(\Omega) \,.
\end{equation}
The following simple lemma is key to our argument.

\begin{lemma}\label{davieslemma}
	For any $x\in\Omega$ and any $\rho>0$,
	$$
	|\Omega\cap B_\rho(x)|\geq  (1- d^{-1} \rho^2 \delta(x)^{-2}) |B_\rho(x)| \,.
	$$
\end{lemma}

\begin{proof}
	We have
	$$
	|\Omega \cap B_\rho(x)| = \int_{\Sph^{d-1}} \int_0^\rho \1_\Omega(x+t\omega) \,t^{d-1}\,dt\,d\omega
	$$
	and clearly, for any $\omega \in\Sph^{d-1}$ with $d_\omega(x)>\rho$, we have $x+t\omega\in\Omega$ for all $t\in(0,\rho)$. Thus,
	\begin{equation}\label{eq:davieslemma}
		|\Omega \cap B_\rho(x)| \geq |\{ \omega\in\Sph^{d-1}:\ d_\omega(x)> \rho \}| d^{-1} \rho^d \,.
	\end{equation}
	On the other hand, clearly, 
	$$
	\rho^{-2} |\{ \omega\in\Sph^{d-1}:\ d_\omega(x)\leq \rho \}| \leq \int_{\Sph^{d-1}} d_\omega(x)^{-2}\,d\omega = d^{-1} |\Sph^{d-1}|\ \delta(x)^{-2} \,,
	$$
	or, equivalently,
	$$
	|\{ \omega\in\Sph^{d-1}:\ d_\omega(x)> \rho \}|
	\geq \left( 1- d^{-1} \rho^2 \delta(x)^{-2} \right) |\Sph^{d-1}| \,.
	$$
	Inserting this bound into \eqref{eq:davieslemma} implies the lemma.
\end{proof}

%%%%%%%%%%%%%

\section{A theorem of Lieb}

Let $-\Delta_\Omega^D$ be the Dirichlet Laplacian in $L^2(\Omega)$ and
\begin{equation}
	\label{eq:ev}
	\lambda_\Omega := \inf\spec (-\Delta_\Omega^D) = 
	\inf\left\{ \int_\Omega |\nabla u|^2\,dx :\ u\in H^1_0(\Omega) \,,\ \int_\Omega |u|^2\,dx = 1 \right\}.
\end{equation}
It is well-known that if $\Omega$ is mean-convex, then $\lambda_\Omega$ is bounded from below by a constant times the inverse square of the radius of the largest ball contained in $\Omega$ and that this is not true for general open $\Omega$. It is a theorem of Lieb \cite{Li} that this remains true for general open $\Omega$, provided `the largest ball contained in $\Omega$' is replaced by `a ball that intersects $\Omega$ significantly'. Here we give a simple alternative proof of this result using \eqref{eq:davies} (albeit with a worse constant)\footnote{Davies's result \cite[Theorem 1.5.11]{Da} and Lieb's \cite{Li} are superior if the supremum in Theorem \ref{lieb} becomes very small.}.

\begin{theorem}\label{lieb}
	Let $\Omega\subset\R^d$ be open. Then for any $\rho>0$,
	$$
	\lambda_\Omega \geq \frac d{4\rho^2} \left( 1- \sup_{x\in\Omega} \frac{|\Omega\cap B_\rho(x)|}{|B_\rho(x)|} \right).
	$$
\end{theorem}

Clearly, this theorem implies for all $0<\theta<1$,
$$
\lambda_\Omega \geq \frac{d(1-\theta)}{4\rho_\theta^2} \,,
\qquad\text{where}\qquad
\rho_\theta := \inf\left\{ \rho>0 :\ \sup_{x\in\Omega} \frac{|\Omega\cap B_\rho(x)|}{|B_\rho(x)|} \leq \theta \right\}.
$$

\begin{proof}
	Inserting \eqref{eq:davies} into \eqref{eq:ev}, we obtain
	$$
	\lambda_\Omega \geq \frac14 \inf\left\{ \int_\Omega \delta^{-2} |u|^2\,dx :\ u\in H^1_0(\Omega) \,,\ \int_\Omega |u|^2\,dx = 1 \right\} \geq \frac14 \inf_\Omega \delta^{-2} \,.
	$$
	Inserting the lower bound on $\delta^{-2}$ from Lemma \ref{davieslemma} we obtain the theorem.
\end{proof}

\emph{Remarks.} (1) The theorem remains valid for the principal eigenvalue of the $p$-Laplacian with $1<p<\infty$. This follows from the validity of the analogue of \eqref{eq:davies} for $1<p<\infty$. Lieb's proof works also in the case $p=1$.\\
(2) If $\lambda$ is an eigenvalue of $-\Delta_\Omega$, then there is an $x\in\Omega$ such that for all $\rho>0$, $\lambda \geq d(4\rho^2)^{-1} (1- |\Omega\cap B_\rho(x)|/|B_\rho(x)|)$. This follows from the same method of proof, by noting that in this case the inequality $\lambda \geq (1/4) \int_\Omega \delta^{-2} |u_0|^2\,dx$ for a normalized eigenfunction $u_0$ implies that there is an $x\in\Omega$ with $\lambda \geq 1/(4\delta(x)^2)$.\\
(3) Lieb's result was improved upon in \cite{MaSh} in the sense that the overlap between $\Omega$ and $B_\rho(x)$ is quantified in terms of capacity instead of measure. It would be interesting to investigate whether there is a strengthening of \eqref{eq:davies} that implies this result.

%%%%%%%%%%%%%%%%%

\section{A theorem of Rozenblum}

We denote by $N_\leq(\lambda,-\Delta_\Omega^D)$ the total spectral multiplicity of $-\Delta_\Omega^D$ in the interval $[0,\lambda]$. It is well-known  \cite{Ro} that for $\Omega$ of finite measure, one has Weyl asymptotics $N_\leq(\lambda,-\Delta_\Omega^D)\sim (2\pi)^{-d}\omega_d |\Omega|\lambda^{d/2}$ as $\lambda\to\infty$, as well as a universal bound $N_\leq(\lambda,-\Delta_\Omega^D)\leq C_d |\Omega|\lambda^{d/2}$ for all $\lambda>0$. A theorem of Rozenblum \cite{Ro} implies, in particular, that sets $\Omega$ that satisfy the reverse inequality $N(\lambda,-\Delta_\Omega^D)\geq \epsilon |\Omega|\lambda^{d/2}$ for some $\lambda>0$ have a substantial `well-structured' component at spatial scale $\lambda^{-1/2}$. 

\begin{theorem}\label{rozenblum}
	For any $\theta\in(0,1]$ there are constants $c_1(\theta),c_2(\theta,d)>0$ with the following property. For any open set $\Omega\subset\R^d$ and any $\lambda>0$ there are disjoint balls $B^{(1)},\ldots,B^{(M)}\subset\R^d$ of radius $c_1 \lambda^{-1/2}$ such that
	$$
	|\Omega\cap B^{(m)}|\geq (1-\theta) |B^{(m)}|
	\qquad\text{for all}\ m=1,\ldots,M
	$$
	and
	$$
	M\geq c_2 \, N_{\leq}(\lambda,-\Delta_\Omega^D) \,.
	$$
\end{theorem}

Note that choosing $\lambda=\lambda_\Omega$ we obtain again Theorem \ref{lieb}, up to constants.

\begin{proof}
	We begin by giving the proof in dimension $d\geq 3$, where we have
	\begin{equation}
		\label{eq:floss}
		N_{\leq}(\lambda,-\Delta_\Omega^D) \leq L_d \int_\Omega \left(\lambda - \frac{1}{4\delta(x)^2}\right)_+^\frac{d}{2}\,dx \,.
	\end{equation}
	This appears in \cite{FrLo}, but a weaker version with $1/4$ replaced by a smaller constant follows easily by \eqref{eq:davieslemma} and the CLR inequality (see \cite{Fr} for references).
	
	Let $E:=\{ x\in\Omega :\ \delta(x) \geq (4\lambda)^{-1/2} \}$. Then, by Lemma \ref{davieslemma},
	$$
	|\Omega\cap B_\rho(x)|\geq (1-4d^{-1} \rho^2\lambda)|B_\rho(x)|
	\qquad\text{for all}\ x\in E \ \text{and all}\ \rho>0 \,.
	$$
	For $\rho=(\theta d/(4\lambda))^{1/2}$ the claimed density condition is satisfied for each such ball.
	
	Let $B_\rho(x_m)$ be a maximal disjoint subcollection of $B_\rho(x)$, $x\in E$. Then $E\subset \bigcup_m B_{2\rho}(x_m)$ (since for any $x\in E$ there is an $x_m$ such that $B_\rho(x)$ intersects $B_\rho(x_m)$, so $|x-x_m|<2\rho$, so $x\in B_{2\rho}(x_m)$). In case there are infinitely many $x_m$ we are done. If there are finitely many $x_m$, say $M$, then
	\begin{align*}
	\int_\Omega \left(\lambda - \frac{1}{4\delta(x)^2}\right)_+^\frac{d}{2}dx & = \int_E \left(\lambda - \frac{1}{4\delta(x)^2}\right)^\frac{d}{2}dx \leq \lambda^\frac d2 |E| \leq \lambda^\frac d2 \sum_m |B_{2\rho}(x_m)| \\
	& = \omega_d\, 2^d\, \lambda^\frac d2\, \rho^d\, M = d^\frac d2\, \omega_d\, \theta^{\frac d2} M \,.
	\end{align*}
	Together with \eqref{eq:floss} this gives the claimed lower bound on $M$ for $d\geq 3$.
	
	For $d=2$ (the case $d=1$ is easy) we bound $N_{\leq}(\lambda,-\Delta_\Omega^D) \leq \lambda^{-\gamma} \Tr(-\Delta_\Omega^D-2\lambda)_-^\gamma$ for any $\gamma>0$ and use the fact \cite{FrLo} that
	$$
	\Tr(-\Delta_\Omega^D-\mu)_-^\gamma \leq L_{\gamma,2} \int_\Omega \left( \mu - \frac{1}{4\delta(x)^2} \right)_+^{\gamma+1}dx \,.
	$$
	The claimed bound now follows similarly as before.	
\end{proof}

\emph{Remarks.} (1) In Rozenblum's formulation, the balls are required to be centered on $(c\lambda^{-1/2})\Z^d$. This can also be achieved by a minor modification of our proof.\\
(2) In fact, Rozenblum proves a stronger theorem where the overlap between $\Omega$ and $B_\rho(x)$ is quantified in terms of capacity instead of measure. It would be interesting to investigate whether there is a corresponding strengthening of \eqref{eq:floss}.\\
(3)  A related result for Schr\"odinger operators was proved in \cite{Fe}.\\
(4) Theorem \ref{rozenblum} might be useful in the problem of maximizing $\Tr(-\Delta_\Omega-\lambda)_-^\gamma$ among sets $\Omega$ of given measure; see \cite{La,FrLa} for partial results for $\gamma\geq 1$.

%%%%%%%%%%%%%%%%%%%%%%%%%%%%%%%%%%%%%%%%%%%%%%%%%%%%%%%%%%%%%%%%%%%%%%%%%%%%%%%%
%%%%%%%%%%%

\bibliographystyle{amsalpha}

\end{document}